\documentclass{amsart}

\usepackage{amsmath,amsthm,amssymb,amsfonts,parskip, mathtools, yfonts}
\usepackage{tipa}
\usepackage[nobysame]{amsrefs}
\usepackage{color}
\usepackage{enumitem}
\usepackage[all]{xy}

\usepackage{endnotes}
\usepackage[mathscr]{eucal}

\theoremstyle{plain}
\newtheorem*{Theorem}{Theorem}
\newtheorem{theorem}{Theorem}[section]
\newtheorem{proposition}[theorem]{Proposition}
\newtheorem{lemma}[theorem]{Lemma}
\newtheorem*{Question}{Question}
\newtheorem{corollary}[theorem]{Corollary}

\newtheorem*{Proposition}{Proposition}

\theoremstyle{definition}
\newtheorem{definition}[theorem]{Definition}
\newtheorem{example}[theorem]{Example}

\theoremstyle{remark}
\newtheorem{remark}[theorem]{Remark}

\numberwithin{equation}{theorem}

\def\lra{\longrightarrow}
\def\Lra{\Longrightarrow}

\DeclarePairedDelimiter{\brackets}{(}{)}
\newcommand{\T}[2]{\operatorname{{\tau}}_{{#2}} (#1)}
\newcommand{\End}[2]{\operatorname{End}_{#1} (#2)}

\renewcommand{\hom}[3]{\operatorname{Hom}_{#1} (#2, #3)}

\newcommand{\IM}{\operatorname{Im} \brackets}

\newcommand{\Ass}{\operatorname{Ass} \brackets}

\newcommand{\Ann}{\operatorname{Ann}_R }

\newcommand{\Ext}[4]{\operatorname{Ext}^{#1}_{#2} (#3, #4)}

\newcommand{\injd}{\operatorname{inj \,dim}_R}

\newcommand{\depth}{\operatorname{depth} }

\begin{document}

\title[]{Trace Ideals and the Gorenstein Property}

\author{Haydee Lindo}
\author{Nina Pande}
\curraddr{Dept. of Mathematics \& Statistics, Williams College, Williamstown, MA, USA}
\email{haydee.m.lindo@williams.edu}
\email{ngp3@williams.edu}

\date{\today}



\maketitle


\begin{abstract}

Let $R$ be a local Noetherian commutative ring. We prove that $R$ is an  Artinian Gorenstein ring if and only if every ideal in $R$ is a trace ideal. We discuss when the trace ideal of a module coincides with its double annihilator. 
\end{abstract}
\section{Introduction}

Let  $R$ be a ring and $M$ an $R$-module. The trace ideal of $M$, denoted $\T R M$, is the ideal generated by the homomorphic images of $M$ in $R$.

The theory of trace ideals has proved useful in various contexts but fundamentally the literature is dominated by two avenues of inquiry. First, given an $R$-module $M$ what does its trace ideal say about $M$? For instance, it is known that trace ideals detect free-summands and that $M$ is projective if and only if its trace ideal is idempotent; see  \cite{MaxOrders,Lam,Whitehead1980,Herberatraceideal}. More recently,  Lindo   discussed the role  of the trace ideal of a module in calculating the center of its endomorphism ring; see  \cite{Lindo1}. Also, Herzog, Hibi, Stamate and Ding have studied the trace ideal of the canonical module to understand deviation from the Gorenstein property in $R$; see \cite{tracecanonical, DingTrace}. 

A second category of question asks: given a ring, what do the characteristics of its class of trace ideals imply about the ring? For example, in \cite{TraceProperty1987}  Fontana, Huckaba and Papick characterize Noetherian domains where every trace ideal is prime; see also  \cite{TraceProperty1987, LucasMcNair2011, LucasRTP}.





This paper addresses both of these questions when $R$ is a local Artinian Gorenstein ring. In this setting, we show that the trace ideal of an $R$-module $M$ coincides with its double annihilator; see Proposition \ref{refl}. In Remark \ref{Art} we recall that all ideals over an Artinian Gorenstein ring are trace ideals. We then show that this property characterizes local Artinian Gorenstein rings; see Theorem \ref{main}. We prove 

\begin{Proposition} 
Let $R$ be a local Artinian Gorenstein ring  and $M$ a finitely generated $R$-module. Then $\T R M= \Ann \Ann M$.
\end{Proposition}

\begin{Theorem}
Let $R$ be a local Noetherian ring with maximal ideal $\mathfrak m$. Then the following are equivalent
\begin{enumerate}[label=(\roman*)]
\item $R$ is an Artinian Gorenstein ring
\item Every ideal is a trace ideal.
\item Every principal ideal is a trace ideal
 
\end{enumerate}

\end{Theorem}

\section{Preliminaries}

Let $R$ be a commutative Noetherian ring and $M$ a finitely generated $R$-module.
The purpose of this section is to define the trace ideal of a module $M$ and relate it to $\Ann \Ann M$.

A trace ideal is a specific type of trace module.
\begin{definition}
Given $R$-modules $M$ and $X$, the trace (module) of $M$ in X is

\begin{align*}
\T X M &:= \displaystyle \sum_{\alpha \in {\hom R M X}} \alpha(X)\\
&\, = \hom R M X M
\end{align*}

where  $\hom R M X M$ denotes the $R$-submodule of $X$ generated by elements of the form $\alpha(m)$ for $\alpha$ in $\hom R M X$ and $m$ in $M$.

 The ideal $\T R M$ is called the trace ideal of $M$ (in $R$).

We say $A$ is a trace module (trace ideal) provided $A = \T X M$ $(=\T R M)$ for some $R$-module M.

\end{definition}

\begin{remark} \label{facts}
Note, an $R$-submodule $M$ in $X$ is a trace module in $X$ if and only if the inclusion $M\subseteq X$ induces an isomorphism $\End R M \cong \hom R M X$. Also, an ideal $I$ in $R$ is a trace ideal only if and only if it is its own trace ideal;  see \cite[Proposition 2.8]{Lindo1}.
\end{remark}

\begin{remark} \label{tracepresmatrix}
One may calculate the trace ideal of a module from its presentation matrix.  Suppose $[M]$ is a presentation matrix for an $R$-module $M$ and $A$ is a matrix whose columns generate the kernel of $[M]^*$, the transpose of $[M]$. Then there is an equality:
\[ \T R M = I_1(A);\]

where $I_1(A)$ is the ideal generated by the entries of $A$; see,   \cite[Remark 3.3]{Vascaffine}.
\end{remark}

\begin{definition}
The annihilator of $M$ (in $R$) is the ideal \[\Ann M := \{r\in R| r M = 0\}.\]
\end{definition}

\begin{lemma} \label{principal}
Let $M$ be a cyclic $R$-module. Then $\T R M = \Ann \Ann M$.
\end{lemma}

\begin{proof}
Set $M = Rm$. The presentation matrix $[M]$ of $M$ is a $1 \times n$ matrix whose entries generate $\Ann m$. Maps $\alpha \in \hom R {M} R$ induce and are induced by $1 \times 1$ matrices $[y] \in \hom R {R} R$   such that  $[y][M]= 0$. These are spanned by the generators of $\Ann \Ann {M}$.

 $$
 \xymatrix{R^n \ar@[->][rrrr]^{\left[\begin{array}{c}M \end{array}\right ]=\left[\begin{array}{cccc} x_1 & x_2 & \cdots & x_n \end{array}\right]} \ar@{-->>}[drr]&&&& R \ar@{->>}[dr]^{\hspace{.2in}\circlearrowleft} \ar@{-->}[rr]^{\left[\begin{array}{c}y \end{array}\right]} && R .\\
&&\Ann m \ar@{^{(}-->}[urr]^{\hspace{-1.75in}\circlearrowleft}&&& Rm \ar@{-->}[ur]_{\alpha}& \\
 }
 $$
 
It follows that $\T R {M} = \Ann \Ann M$.
\end{proof}
\pagebreak
\begin{corollary}
Let $M$ be a finitely generated $R$-module. Then $\T R M \subseteq \Ann \Ann M$. 
\end{corollary}

\begin{proof}
Let $\{ m_1, \ldots, m_n\}$ be a generating set for $M$.  For each $\alpha$ in $\hom R M R$, $\alpha (M) = \displaystyle \sum_{i = 1}^n\alpha (Rm_i)$. By Lemma \ref{principal} it follows that
\begin{align*}
\T M R &\subseteq \sum_{i=1}^n  \T R {Rm_i} \\
&= \sum_{i=1}^n  \Ann \Ann {m_i}\\ 
& \subseteq \Ann \Ann M   \qedhere \end{align*}
\end{proof}

\begin{remark}
We show $\T M R = \Ann \Ann M$ when $R$ is Artinian Gorenstein; see Proposition \ref{refl}.
\end{remark}

\begin{lemma} \label{lem}
Given an ideal $I$ in $R$, there is an equality $I = \Ann \Ann I$ if and only if $I = \Ann J$ for some ideal $J$. 
\end{lemma}

\begin{proof}
Taking $J = \Ann I$ yields the forward implication. Given $I = \Ann J$ for some ideal $J$, the backwards implication follows from the equality  \[\Ann \Ann \Ann J  = \Ann J. \qedhere \]
\end{proof}

\begin{corollary} \label{Anntrace}
Given an ideal $I$ in $R$, if $I = \Ann \Ann I$ then $I$ is a trace ideal. As a result, given an ideal J, $I= \Ann J$ is a trace ideal. 
\end{corollary}

\begin{proof}
The first statement follows immediately from the containments \[I \subseteq \T R I \subseteq \Ann \Ann I.\] The second statement follows from the first and Lemma \ref{lem}.
\end{proof}

\begin{example}
Consider $R = k[x,y]_{(x,y)}/(x^2, xy)$ for some field $k$. Note $R$ has depth zero and Krull dimension one. The ideal $(x)$ is its own  trace ideal since $\Ann \Ann {(x)} = (x)$. The ideal $(y)$ is not a trace ideal since $\Ann \Ann {(y)} = (x,y)$.
\end{example}
\section{Main Results}
In this section $R$ is a local Noetherian commutative ring. We identify the trace ideals of modules over Artinian Gorenstein rings as their double annihilator and characterize local Artinian Gorenstein rings in terms of their classes of trace ideals. 

Recall  \cite[Theorem 18.1]{Mats}. In particular,
\begin{theorem} \label{gorenstein}
Let $(R, \mathfrak m)$ be a local Noetherian ring of Krull dimension $d$ with residue field $k$. Then the following are equivalent
\begin{enumerate}[label=(\roman*)]
\item $R$ is Gorenstein;
\item $\injd R = d$;
\item $\depth R = d$ and $\Ext d R {k} R \cong k$. \qed
\end{enumerate} 
\end{theorem}

\begin{remark}\label{Art}
There are several arguments showing that all ideals in a local Artinian Gorenstein ring are trace ideals:
\begin{enumerate}[label=(\roman*)]
 \item Given an ideal  $I$ in $R$,  one such argument considers the exact sequence \[0 \lra I \lra R \lra R/I \lra 0.\]
Applying $\hom R {\_} R$ yields the top exact sequence below
 $$
 \xymatrix{ \cdots \ar@[->][r] & \hom R R R \ar@{->}[r]  \ar@[->][d]^{\cong} & \hom R I R \ar@[->][r] \ar@[->][d]_{=} & \Ext 1 R {R/I} R \ar@[->][d]_{=} \ar@[->][r]& \cdots \\
&R \ar@{->>}[r]&\hom R I R \ar@{->}[r]&0& \\
 }
 $$
 
 where $\Ext 1 R {R/I} R$= 0 because $R$ is self-injective. As a result, all maps $\alpha$ from $I$ to $R$ are given my multiplication by some element $r$ in $R$. Therefore, $I$ is its own trace ideal. 
 
\item A second argument is found in the proof of Proposition 1.2 in \cite{BrandtChar}. Here Brandt shows that $M$ being a trace module in $X$ implies  that $M$ is an $\End R X$-submodule of $X$ and that the converse holds when $X$ is injective. In particular, when $R$ is self-injective the trace ideals of $R$ are precisely the $R$-submodules of $R$, that is, the ideals. 

$(\Lra)$ Recall $\hom R M X$ is an $\End R X$-module. Thus
\begin{align*}
\End R X \T X M & = \End R X \hom R M X M\\
& = \hom R M X M\\
& = \T X M
\end{align*}

$(\Longleftarrow)$
Say $i$ is the inclusion $M \subseteq X $ and $\phi$ is any map in $\hom R M X$. Since $X$ is injective, there exists $\bar \phi$ in $\End R X$ such that $\bar \phi i  = \phi$.  By assumption $M$ is an $\End R X$-module, so that $\phi (M) =\bar \phi i (M)=  \bar \phi|_{M} (M) \subseteq M$. Therefore $M$ is a trace module in $X$; see Remark \ref{facts}. 

\item A third argument proceeds from Corollary \ref{Anntrace} and  Lemma \ref{GorAnn} below.
\end{enumerate}
\end{remark}
The following characterization of local Artinian Gorenstein rings is well-known; see, for example, Exercise 3.2.15 in \cite{BandH}. 

\begin{lemma} \label{GorAnn}
Let $R$ be a local Artinian commutative ring.  Then $R$ is a Gorenstein ring if and only if $I = \Ann \Ann I$ for every ideal $I$ of $R$. \qed 
 \end{lemma}

\begin{proposition} \label{refl}
Let $R$ be a local Artinian Gorenstein ring  and $M$ a finitely generated $R$-module. Then $\T R M = \Ann \Ann M$.
\end{proposition}

\begin{proof}
Every finitely generated module over an Artinian Gorenstein ring is reflexive  and $M$ being reflexive implies $\Ann M = \Ann {\T R M}$; see {\cite[Corollary 2.3 ]{Vasc1}} and  \cite[Proposition 2.8 (vii)]{Lindo1}. Also, since $R$ is Artinian Gorenstein, by Lemma \ref{GorAnn} one has $I = \Ann \Ann I$ for all ideals $I \subseteq R$.  It follows that \[\Ann \Ann M = \Ann \Ann {\T R M} = \T R M. \qedhere\] \end{proof}

\begin{theorem} \label{main}
Let $R$ be a local Noetherian ring with maximal ideal $\mathfrak m$. Then the following are equivalent
\begin{enumerate}[label=(\roman*)]
\item $R$ is an Artinian Gorenstein ring;
\item Every ideal is a trace ideal;
\item Every principal ideal is a trace ideal.
 
\end{enumerate}

\end{theorem}

\begin{proof}
If $R$ is Artinian Gorenstein then $I = \Ann \Ann I$ for each ideal $I$ in $R$; see Lemma \ref{GorAnn}. By Corollary \ref{Anntrace} every ideal is a trace ideal and, in particular, every principal ideal is a trace ideal.

Now assume every principal ideal is a trace ideal.  For each $r$ in $R$ one has \[{(r) = \Ann \Ann {(r)}};\] see Lemma \ref{principal}.  Therefore $r$ is a zerodivisor, $\depth R =0$ and $\mathfrak m \in \Ass R$.  

Recall that the nilradical of a ring is the intersection of its minimal primes. Since $\depth R  = 0$, if  $\dim R >0$ then there exists a zerodivisor $x$ in $R$ which is not nilpotent. For all $n \in \mathbb{N}$, $\Ann {(x^n)}$ is nonzero and contained in $\mathfrak m$. Therefore \[ \Ann \mathfrak m \subseteq \Ann \Ann {(x^n)} = (x^n). \]
That is $\Ann \mathfrak m \subseteq \cap_{n\in \mathbb{N}} (x^n)$ and so $\Ann \mathfrak m = 0$ by the Krull Intersection Theorem \cite[Corollary 5.4]{EisenbudCommAlg}. This is a contradiction because $\mathfrak m \in \Ass R$. Thus $\dim R =0$.

As a zero-dimensional Cohen-Macaulay ring, the socle of $R$ is the sum of the finite number of minimal nonzero ideals, each isomorphic to $k =R/\mathfrak m$. Since each minimal nonzero ideal is also a trace ideal, the socle of $R$ is isomorphic to $k$. Therefore $R$ is Artinian and Gorenstein. 
\end{proof}

\begin{remark} Given an Artinian ring $R$, one commonly determines if $R$ is Gorenstein by checking if its socle is one-generated over $R$. This is equivalent to checking that $k$ is a trace ideal in $R$. As a consequence of Theorem \ref{main}, one can use any ideal to check if $R$ is Gorenstein. In practice, given an Artinian ring $R$, $R$ is not Gorenstein if there exists an ideal $I$ in $R$ and a map $\alpha \in \hom R I R$ such that $\IM \alpha \not\subseteq I$. 
\end{remark}

\begin{example}
Consider the subring $S = k[x^4, x^3y, xy^3,y^4] \subset k[x,y]$ for some field $k$. Set $R = k[x^4, x^3y, xy^3,y^4]/(x^4, y^4)$. Then $R$ is not Gorenstein because there exists an  $R$-homomorphism 
$$
\xymatrix @R=.1pc{(x^3y)\ar@{->}[r] & (xy^3)\\
x^3y \ar@{|->}[r] & xy^3.
}
$$

whose image is not contained in $(x^3y)$. 
\end{example}

\begin{remark}
It is known that all ideals of grade greater than or equal to 2 are trace ideals, as are all ideals in local Artinian Gorenstein rings; see Remark 2.3 in \cite{Lindo1} and Remark \ref{Art} above. Recently, a conjecture of Huneke and Wiegand has been verified for modules isomorphic to trace ideals in one dimensional Gorenstein domains; see \cite[Proposition 6.8]{Lindo1}. However, an ideal may be isomorphic to a trace ideal without being a trace ideal itself. For example consider the ideal $I = (xy, xz)$ in $k[x,y,z]$, for some field $k$, where $\T R I = (y,z)$. This investigation leads naturally to the following open questions:

\end{remark}
\begin{Question}
In which rings is every ideal isomorphic to a trace ideal?
\end{Question}

\begin{Question}
What is the class of modules isomorphic to trace ideals over \\{one-dimensional} Gorenstein domains?
 \end{Question}
 
 
 

\section*{Acknowledgements}

Special thanks to Susan Loepp. Thanks also to Andrew Bydlon, Peder Thompson, Graham Leuschke, Ivan Martino and Anthony Iarrobino for several useful discussions. 

\bibliographystyle{amsplain}	


\end{document}